\documentclass[12pt]{article}

\usepackage{mathtools, amsthm, accents, tikz, amssymb, latexsym, amsmath, bm, amscd, amsfonts, array, lmodern, enumerate, stmaryrd, rotating, caption, graphicx, hyperref, float,tikz-cd, color, yfonts, stmaryrd, bookmark}
\usepackage[new]{old-arrows}
\usepackage[outline]{contour}
\usetikzlibrary{calc}
\usepackage[all]{xy}
\CompileMatrices
\hypersetup{nesting=true,debug=true,naturalnames=true}

\def\UXnpq{\UDF(\puz_{p,q},n)}
\def\cc{\hspace{.8mm},\hspace{.4mm}}
\def\puz{\protect\operatorname{Puz}}
\def\hdim{\protect\operatorname{hdim}}
\def\sal{\protect\operatorname{Sal}}
\def\raag{\protect\operatorname{RAAG}}

\def\UConf{\protect\operatorname{UConf}}
\def\DF{\protect\operatorname{DConf}}
\def\UDF{\protect\operatorname{UDConf}}
\def\UF{\protect\operatorname{UConf}}

\def\C{\protect\operatorname{C}}
\def\UC{\protect\operatorname{UC}}

\def\hdim{\protect\operatorname{hdim}}

\newtheorem{proposition}{Proposition}[section]
\newtheorem{corollary}[proposition]{Corollary}
\newtheorem{definition}[proposition]{Definition}
\newtheorem{theorem}[proposition]{Theorem}
\newtheorem{remark}[proposition]{Remark}
\newtheorem{example}[proposition]{Example}
\newtheorem{lemma}[proposition]{Lemma}
\newtheorem{examples}[proposition]{Examples}

\begin{document}

\title{Square-section braid groups and Higman-Neumann-Neumann extensions}

\author{Omar Alvarado-Gardu\~no and Jes\'us Gonz\'alez}

\date{\today}

\maketitle

\begin{abstract}
For positive integers $n$, $p$ and $q$ with $pq-n>0$, let $\UC(n,p\times q)$ denote the configuration space of $n$ unlabelled hard unit squares in the rectangle $[0,p]\times[0,q]$, and let $B_n(p\times q)$ denote the corresponding fundamental group. It is known that, as $pq-n$ becomes large, $\UC(n,p\times q)$ starts capturing homotopical properties of the classical configuration space of $n$ unlabelled pairwise-distinct points in the plane. At the start of this approximation process, $\UC(pq-1,p\times q)$ is homotopy equivalent to a wedge of $(p-1)(q-1)$ circles, while the only other general families of spaces $\UC(n,p\times q)$ known to be aspherical are $\UC(n,p\times2)$ for $p\geq n$, and $\UC(pq-2,p\times q)$. The fundamental groups of the former family are known to be responsible for the ``right-angled'' relations in Artin's classical braid groups. We prove that the fundamental groups of the latter family have a minimal presentation all whose relators are commutators. In particular, after explaining how $B_{2p-2}(p\times2)$ arises as the right-angled Artin group (RAAG) associated to a certain meta-edge, we show that $B_{3p-2}(p\times3)$ is a Higman-Neumann-Neumann extension of the RAAG associated to the corresponding meta-square. We provide a geometric interpretation of the latter fact in terms of Salvetti complexes.
\end{abstract}



\section{Introduction}
Configuration spaces of points are pivotal in Mathematics and their applications across science. From a purely theoretical perspective, much of their interest originates from a well known theorem of Fox and Neuwirth (\cite{MR150755}) stating that the classical configuration space $\UConf(\mathbb{R}^2,n)$ of $n$ unlabelled distinct points in the plane classifies Artin's braid group $B_n$ on $n$ strands. More recently, there has been an increasing interest in systems of particles with actual geometry due to their applications outside pure mathematics. For instance, in topological robotics, configuration spaces of geometric particles (CSGP) have been used in the design of algorithms for collision-free motion planning of autonomous agents in constrained environments (\cite{robotics}). Likewise, in soft condensed matter, CSGP have applications in entropy-driven self-assembly processes (\cite{sq2}) whereas, in statistical mechanics, physicists have studied CSGP to gain knowledge about how specific changes in system parameters (like packing density) correspond to topological shifts and phase transitions in matter. 

\begin{remark}\label{melting}{\em
By examining structural and thermodynamic properties as functions of the packing density, physicists have identified different phases in CSGP. In particular, configuration spaces of 2-dimensional hard (i.e.~non-overlapping) objects have been shown to undergo through sophisticated melting processes (\cite{PhysRevX.7.021001}). In the case of hard squares (modeled in part by Definition \ref{hscsdef} below), a solid phase is characterized by a widespread 4-fold symmetric crystalline structure whereas, in liquid and gas phases, there is a lack of spatial periodicity. Now, while no sharp difference has been observed between a gas and a liquid phase, the melting process is particularly interesting as it involves in fact three phases, where a tetratic phase holding in a narrow density window separates a solid phase from an isotropic liquid phase. In this relatively brief tetratic phase, hard squares can drift but they still assemble structures with a 4-fold symmetry that decays slowly with distance.
}\end{remark}

Motivated by their subtle topological aspects, mathematicians have began to study CSGP (\cite{MR4096335, MR4640135, MR4298668, MR4749152, MR4718125,  MR2826926, MR3889263, MR4299671}). The configuration spaces we address in this paper were first studied systematically in mathematical terms by Alpert in \cite{MR4096335}.

\begin{definition}[Section 2 of \cite{MR4096335}]\label{hscsdef}
For positive integers $n,p,q$ with $n\leq pq$, $\C(n,p\times q)$ stands for the configuration space of $n$ labelled closed non-overlapping unit squares, whose sides are parallel to the canonical axes, and lie inside the rectangle $R_{p,q}:=[0,p]\times[0,q]$. We stress the fact that the non-overlapping condition \emph{does not} exclude the possibility that unit squares touch each other at their boundaries, or touch the boundary of $R_{p,q}$ ---so, in physics terminology, these are \emph{hard} squares. The $n$-th symmetric group acts freely on $\C(n,p\times q)$ by permutation of the labelled squares, and the corresponding orbit space, denoted by $\UC(n,p\times q)$, is referred to as the unlabelled configuration space of $n$ hard squares in $R_{p,q}$.
\end{definition}

Note that $\UC(n,p\times q)$ is contractible when $\min(p,q)=1$. We thus assume $p\geq q\geq2$ from now on, so that $\UC(n,p\times q)$ is path-connected (\cite[Subsection 3.2]{AGK}). As explained in the introductory section of \cite{AGK}, the fundamental group
$$
B_n(p\times q):=\pi_1(\UC(n,p\times q))
$$
can be thought of as a $(p,q)$-approximation of $B_n:=\pi_1(\UConf(\mathbb{R}^2,n))$. It has been observed that the spaces $\UC(n,p\times q)$ would in general fail to be aspherical (\cite[Theorem 3.2]{plachtaunpublished}), yet a few notable exceptions are known (Proposition \ref{asphericalspaces} below), and we focus on describing their homotopy types through the corresponding fundamental groups $B_n(p\times q)$.

\begin{proposition}[{\cite[Theorems 1.4 and 1.5]{AGK}, \cite[page 2597]{MR4640135} and \cite[Theorem 3.6]{MR4298668}}]\label{asphericalspaces} Spaces $\UC(n,p\times q)$ of the following three types are aspherical:
\begin{itemize}
\item[{\em (1)}] $\UC(n,p\times2)$ for $n\leq p$.
\item[{\em (2)}] $\UC(pq-1,p\times q)$.
\item[{\em (3)}] $\UC(pq-2,p\times q)$.
\end{itemize}
\end{proposition}

The homotopy type of spaces of types (1) and (2) in Proposition \ref{asphericalspaces} have recently been described (see Remark \ref{first2topologies} below). The main result in this paper is a description of the homotopy type of all spaces of type (3). Such a task is attained in Theorem \ref{resultadocasiprincipal} below through an explicit presentation of the corresponding fundamental group. The following standard constructions will be convenient for our purposes. Details can be found in \cite{MR2322545}.

\begin{definition}\label{raag-sal}
For a simple connected graph $\Gamma$, let $\Delta(\Gamma)$, $\sal(\Gamma)$ and $\raag(\Gamma)$ stand, respectively, for the flag (clique) complex, the Salvetti complex and the right-angled Artin group (RAAG) determined by $\Gamma$. Thus $\sal(\Gamma)$ is the polyhedral power $(S^1)^{\Delta(\Gamma)}$, a classifying space for $\raag(\Gamma)$.
\end{definition}

\begin{remark}\label{first2topologies}{\em
The description of the homotopy type of a space of type (2) in Proposition~\ref{asphericalspaces} is straightforward. Namely, let $N_\ell$ stand for the graph with $\ell$ vertices and no edges. Then, as shown in \cite[Theorem 1.4]{AGK}, $B_{pq-1}(p\times q)\cong\raag(N_{(p-1)(q-1)})$, the free non-abelian group of rank $(p-1)(q-1)$, so $\UC(pq-1,p\times q)\simeq\sal(N_{(p-1)(q-1)})$, the wedge of $(p-1)(q-1)$ copies of the circle $S^1$. The homotopy type of a space of type (1) in Proposition~\ref{asphericalspaces} can be given on similar grounds, though details are slightly more involved due to the presence of higher dimensional cells. Indeed, for $n\leq p$, $B_n(p\times2)\cong\raag(G_n)$ so $\UC(n,p\times2)\simeq \sal(G_n)$, where $G_n$ is the graph obtained from the complete graph on $n-1$ vertices after removing the edges of a maximal linear tree (\cite[Theorem 1.1]{AGK}). In particular, the homotopy dimension of $\UC_n(p\times2)$ is $\hdim(\UC_n(p\times2))=\lfloor n/2\rfloor$.
}\end{remark}

%\begin{figure}[h!]
%        \centering
%        \includegraphics[width=0.35\textwidth]{Bk}
%\caption{The bipartite graph $E_k$.}
%        \label{Ek}
%    \end{figure}

\begin{figure}[h!]
        \centering
$$\begin{tikzpicture}[x=.6cm,y=.6cm]
\draw[ultra thin](0,0)--(0,1.5);\draw[ultra thin](0,0)--(2,1.5);
\draw[ultra thin](0,0)--(7,1.5);\draw[ultra thin](0,0)--(9,1.5);
\draw[thin](2,0)--(2,1.5);\draw[thin](2,0)--(7,1.5);\draw[thin](2,0)--(9,1.5);
\draw[thick](7,0)--(7,1.5);\draw[thick](7,0)--(9,1.5);
\draw[very thick](9,0)--(9,1.5);
\node at (0,0) {\scriptsize$\bullet$};\node at (2,0) {\scriptsize$\bullet$};
\node at (7,0) {\scriptsize$\bullet$};\node at (9,0) {\scriptsize$\bullet$};
\node at (0,1.5) {\scriptsize$\bullet$};\node at (2,1.5) {\scriptsize$\bullet$};
\node at (7,1.5) {\scriptsize$\bullet$};\node at (9,1.5) {\scriptsize$\bullet$};
\node[below] at (0,0) {$v_k$};\node[below] at (2,0) {$v_{k-1}$};
\node[below] at (7,0) {$v_2$};\node[below] at (9,0) {$v_1$};
\node[above] at (0,1.5) {$u_1$};\node[above] at (2,1.5) {$u_2$};
\node[above] at (7,1.5) {$u_{k-1}$};\node[above] at (9,1.5) {$u_k$};
\node at (4.5,0) {$\cdots$};\node at (4.5,1.5) {$\cdots$};
\end{tikzpicture}$$
\caption{The bipartite graph $E_k$. Edge thickness suggests ``$v$-families'' of edges. Since $u_i$ and $v_i$ are vertices of degree $i$, these families are dually reorganized by turning $E_k$ upside down.}
        \label{Ek}
    \end{figure}
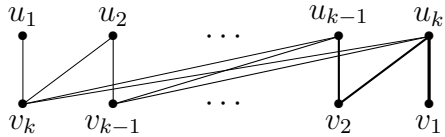

\begin{remark}\label{q2case}{\em
The homotopy type of a space of type (3) in Proposition \ref{asphericalspaces} is known for $q=2$ (\cite[Theorem 1.5]{AGK}). We review the answer for comparison purposes. For a non-negative integer $k$, let $E_k$ stand for the bipartite graph depicted in Figure~\ref{Ek} with vertices $u_i$ and $v_i$, for $1\leq i\leq k$, and an edge $e_{i,j}$ joining $u_i$ and $v_j$ whenever $i+j>k$. Additionally, for a graph $\Gamma$, let $k+\Gamma$ stand for the graph obtained by adding $k$ isolated vertices to $\Gamma$. For instance, $E_0=\varnothing$ and $0+\Gamma=\Gamma$. Then $B_2(2\times2)\cong\raag(N_1) =\mathbb{Z}$, so $\UC(2,2\times2)\simeq\sal(N_1)=S^1$ while, for $p\geq3$, $B_{2p-2}(p\times2)\simeq\raag(3+E_{p-3})$. In particular $\hdim(\UC(2p-2,p\times2))=2$ with $\UC(2p-2, p\times2)$ a union of 2-torii. For instance $\UC(6,4\times2)\simeq(S^1\times S^1)\vee \UC(4,3\times2)$ and $\UC(4,3\times2)\simeq \bigvee_3S^1$.
}\end{remark}

\begin{remark}\label{tetraticphase}{\em
As shown in \cite[Theorem 1.1]{AGK}, the isomorphism $B_n(p\times2)\cong\raag(G_n)$ holds true in the extended range $n\leq2p-5$ (though asphericity of $\UC(n,p\times2)$ is not asserted in the extended range). This leaves us with a situation that matches, at the fundamental-group level, the lack of sharp differences between liquid and gas phases, as well as the existence of a brief tetratic phase (Remark~\ref{melting}). Indeed, all spaces $\UC(n,p\times2)$ with $n\leq2p-5$ are $\pi_1$-isomorphic, and are thus said to have a \emph{($\pi_1$-indistinguishable) liquid-gas topology.} The few remaining spaces $\UC(n,p\times2)$ with $2p-4\leq n\leq2p-1$ are then said to have a \emph{($\pi_1$-indistinguishable) tetratic topology.} As reviewed above, the first two tetratic topologies, namely the cases with $n=2p-1$ and $n=2p-2$, are encoded by the graphs $N_{p-1}$ and $3+E_{p-3}$, respectively, through RAAG-Sal constructions. The point here is that, while Remark~\ref{first2topologies} says that the description of $B_{2p-1}(p\times2)$ extends in a straightforward way to $B_{pq-1}(p\times q)$, our main results (Theorems~\ref{resultadocasiprincipal} and \ref{resultadoprincipal} below) show that the $q\geq3$ case of the second tetratic topology (i.e.~spaces of type (3) in Proposition~\ref{asphericalspaces}) retains only a subtle amount of the RAAG features we have reviewed in this introductory section. Details are spelled out in the next section.
}\end{remark}

\section{Main results}
\begin{definition}
A group is said to be \emph{simple commutator-related} provided it admits a finite presentation whose relators are commutators $[a,b]:=aba^{-1}b^{-1}$ of elements $a$ and $b$ ---neither of which is assumed to be a generator in the given presentation. Such a presentation is said to be a \emph{simple commutator-related structure} of $G$.
\end{definition}

Note that a RAAG is nothing but a group with a simple commutator-related structure whose relators are commutators of actual generators in the given structure.

\begin{theorem}\label{resultadocasiprincipal}
For $p,q\geq3$, the aspherical space $\UC(pq-2,p\times q)$ has torsion-free homology and homotopy dimension given by
$$
\hdim(\UC(pq-2,p\times q))=\begin{cases}
1, & p=q=3; \\
2, & \text{otherwise.}
\end{cases}
$$
The fundamental group $B_{pq-2}(p\times q)$ has a simple commutator-related structure with $\beta_1$ generators and $\beta_2$ relators ---we stress that all relators are commutators--- where $\beta_1:=(p-1)(q-1)+1$ and
$$
\beta_2:=\frac{(p^2+1)(q^2+1)-pq(2p+2q+3)+7(p+q-1)}2
$$
are, respectively, the first and second Betti numbers of $\UC(pq-2,p\times q)$.
\end{theorem}

Generators and relators for the presentation of $B_{pq-2}(p\times q)$ in Theorem \ref{resultadocasiprincipal} are spelled out in Section \ref{structure}. Although technical, the presentation is as efficient as possible.

\begin{corollary}\label{eficiencia}
The presentation of $B_{pq-2}(p\times q)$ in Theorem \ref{resultadocasiprincipal} has the minimal possible number of generators and relators. Explicitly, any finite presentation (whether or not of simple commutator-related type) of $B_{pq-2}(p\times q)$ must have at least $\beta_1$ generators and $\beta_2$ relators.
\end{corollary}
\begin{proof}
Given the torsion-free assertion in Theorem \ref{resultadocasiprincipal}, this is a standard consequence of the Hurewicz theorem in dimension 1 and of Hopf's formula (see Exercise 5 at the bottom page 46 of \cite{MR672956}).
\end{proof}

Tietze transformations can be applied to a given simple commutator-related group in the hope of identifying a new basis in terms of which a RAAG structure arises. This is what happens in all instances reviewed in Remark \ref{q2case}, as well as in a few additional instances with $\min(p,q)=3$ of Theorem \ref{resultadocasiprincipal} (see Examples \ref{stillraags} below). Although the idea fails to identify a RAAG structure in $B_{pq-2}(p\times q)$ for general $p,q\geq3$, it leads us to a description (Theorem~\ref{resultadoprincipal} below) of $B_{pq-2}(p\times q)$, for $\min(p,q)=3$, as an HNN extension of the RAAG associated to a (literally) squared version of the graph in Figure~\ref{Ek}. Before spelling out the answer, it is convenient to illustrate the situation in a few exceptional cases. In what follows we write $G=F_k$ to mean that $G$ is a free non-abelian group of rank $k$, while the more specific notation $G=F(x_1,\ldots, x_k)$ is used to indicate a set of generators.

\begin{examples}\label{stillraags}{\em
Recall from Remark \ref{q2case} the isomorphism $B_{2p-2}(p\times2)\cong\raag\left(3+E_{p-3}\right)$, valid for $p\geq3$. Concerning $B_{3p-2}(p\times3)$, the methods in this paper yield:
\begin{enumerate}
\item For $p=3$, $B_7(3\times3)\cong\raag(5)=F_5$, while $B_4(3\times2)\cong\raag(3)=F_3$.
\item For $p=4$, $B_{10}(4\times3)\cong\raag\left(\raisebox{-2.4mm}{$\hspace{-1mm}\begin{tikzpicture}[x=.7cm,y=.7cm]
\draw(0,0)--(0,.5)--(.5,.5)--(.5,0)--(0,0);
\node at (0,0) {\tiny$\bullet$};\node at (.5,0) {\tiny$\bullet$};
\node at (0,.5) {\tiny$\bullet$};\node at (.5,.5) {\tiny$\bullet$};
\end{tikzpicture}$}+\,3\right)$, while $B_6(4\times2)=\raag\left(\raisebox{-2.4mm}{$\hspace{-.6mm}\begin{tikzpicture}[x=.7cm,y=.7cm]
\draw(0,0)--(0,.5);
\node at (0,0) {\tiny$\bullet$};\node at (0,.5) {\tiny$\bullet$};
\end{tikzpicture}$}+\,3\right)$.
\item For $p=5$, $B_{13}(5\times3)\cong\raag\left(\raisebox{-4.5mm}{$\hspace{-1mm}\begin{tikzpicture}[x=1.2cm,y=1.2cm]
\draw(0+.06,0-.06)--(0+.06,.5+.06)--(0-.06,0+.06)--(0-.06,.5-.06);
\draw(.5+.06,0+.06)--(.5+.06,.5-.06)--(.5-.06,0-.06)--(.5-.06,.5+.06);
\draw(.5-.06,.5+.06)--(0+.06,.5+.06)--(.5+.06,.5-.06)--(0-.06,.5-.06);
\draw(.5+.06,0+.06)--(0-.06,0+.06)--(.5-.06,0-.06)--(0+.06,0-.06);
\node at (0+.06,0-.06) {\tiny$\bullet$};\node at (0+.06,.5+.06) {\tiny$\bullet$};
\node at (0-.06,0+.06) {\tiny$\bullet$};\node at (0-.06,.5-.06) {\tiny$\bullet$};
\node at (.5+.06,0+.06) {\tiny$\bullet$};\node at (.5+.06,.5-.06) {\tiny$\bullet$};
\node at (.5-.06,0-.06) {\tiny$\bullet$};\node at (.5-.06,.5+.06) {\tiny$\bullet$};
\end{tikzpicture}$}+\,1\right)$, while $B_8(5\times2)=\raag\left(\raisebox{-4.2mm}{$\hspace{-1mm}\begin{tikzpicture}[x=1cm,y=1cm]
\draw(0+.06,0-.06)--(0+.06,.5+.06)--(0-.06,0+.06)--(0-.06,.5-.06);
\node at (0+.06,0-.06) {\tiny$\bullet$};\node at (0+.06,.5+.06) {\tiny$\bullet$};
\node at (0-.06,0+.06) {\tiny$\bullet$};\node at (0-.06,.5-.06) {\tiny$\bullet$};
\end{tikzpicture}$}\hspace{-.6mm}+\,3\right)$.
\end{enumerate}
}\end{examples}

Item 1 is just too small, but items 2 and 3 already suggest the key pattern. Namely, if we think of the graph $E_{p-3}$ in the case $q=2$ as a ``$p$-meta-edge'', then the graph relevant for $q=3$ would be a ``$p$-meta-square''. Yet, some subtleties are still missing in Examples \ref{stillraags}. Namely, there is indeed a generator $v_p$ of $B_{3p-2}(p\times3)$ corresponding to the copy of $F_1$ that splits off from $B_{13}(5\times3)$. However, for $p\geq6$, this generator turns out to be related to several of the remaining generators. So, it is more illuminating to write item 3 in Examples \ref{stillraags} in the form
$$
B_{13}(5\times3)\cong\raag\left(\raisebox{-4.5mm}{$\hspace{-1mm}\begin{tikzpicture}[x=1.2cm,y=1.2cm]
\draw(0+.06,0-.06)--(0+.06,.5+.06)--(0-.06,0+.06)--(0-.06,.5-.06);
\draw(.5+.06,0+.06)--(.5+.06,.5-.06)--(.5-.06,0-.06)--(.5-.06,.5+.06);
\draw(.5-.06,.5+.06)--(0+.06,.5+.06)--(.5+.06,.5-.06)--(0-.06,.5-.06);
\draw(.5+.06,0+.06)--(0-.06,0+.06)--(.5-.06,0-.06)--(0+.06,0-.06);
\node at (0+.06,0-.06) {\tiny$\bullet$};\node at (0+.06,.5+.06) {\tiny$\bullet$};
\node at (0-.06,0+.06) {\tiny$\bullet$};\node at (0-.06,.5-.06) {\tiny$\bullet$};
\node at (.5+.06,0+.06) {\tiny$\bullet$};\node at (.5+.06,.5-.06) {\tiny$\bullet$};
\node at (.5-.06,0-.06) {\tiny$\bullet$};\node at (.5-.06,.5+.06) {\tiny$\bullet$};
\end{tikzpicture}$}\right)\star \varphi_5\hspace{.3mm},
$$
namely, as the (trivial) HNN extension of a RAAG resulting from adding a stable letter $v_5$ that conjugates the trivial subgroup $1$ to itself via $\varphi_5\colon 1\to1$. Cases $p=6$ in Figure~\ref{p6} and $p=7$ in Figure~\ref{p7} preserve this form, except that the subgroups conjugated by $v_p$ are no longer trivial but free. All cases $p\geq8$ have fully general characteristics. The case $p=8$ is depicted in Figure \ref{p8}.

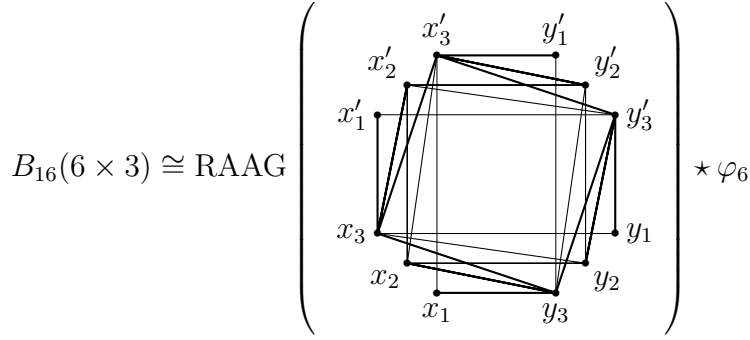
\begin{figure}[h!]
\centering
$B_{16}(6\times3)\cong\raag\left(\raisebox{-2.1cm}{
\begin{tikzpicture}[x=.79cm,y=.79cm]
\draw[thick](-.5,2.5)--(-.5,.5)--(0,3)--(-.5,.5)--(.5,3.5);
\draw(0,0)--(0,3)--(0,0)--(.5,3.5);
\draw[ultra thin](.5,-.5)--(.5,3.5);
\draw[thick](2.5,3.5)--(.5,3.5)--(3,3)--(.5,3.5)--(3.5,2.5);
\draw(0,3)--(3,3)--(0,3)--(3.5,2.5);
\draw[ultra thin](-.5,2.5)--(3.5,2.5);
\draw[thick](3.5,.5)--(3.5,2.5)--(3,0)--(3.5,2.5)--(2.5,-.5);
\draw(3,3)--(3,0)--(3,3)--(2.5,-.5);
\draw[ultra thin](2.5,3.5)--(2.5,-.5);
\draw[thick](.5,-.5)--(2.5,-.5)--(0,0)--(2.5,-.5)--(-.5,.5);
\draw(3,0)--(0,0)--(3,0)--(-.5,.5);
\draw[ultra thin](3.5,.5)--(-.5,.5);
\node at (0,0) {\tiny$\bullet$};\node at (0,3) {\tiny$\bullet$};
\node at (3,0) {\tiny$\bullet$};\node at (3,3) {\tiny$\bullet$};
\node at (0+.5,0-.5) {\tiny$\bullet$};\node at (0-.5,0+.5) {\tiny$\bullet$};
\node at (0+.5,3+.5) {\tiny$\bullet$};\node at (0-.5,3-.5) {\tiny$\bullet$};
\node at (3+.5,0+.5) {\tiny$\bullet$};\node at (3-.5,0-.5) {\tiny$\bullet$};
\node at (3+.5,3-.5) {\tiny$\bullet$};\node at (3-.5,3+.5) {\tiny$\bullet$};
\node[below] at (.5,-.5) {$x_1$};\node[below left] at (0.07,0.07) {$x_2$};\node[left] at (0-.5,0+.5) {$x_3$};
\node[left] at (-.485,2.5) {$x'_1$};\node[above left] at (0.02,2.94) {$x'_2$};\node[above] at (.5,3.5) {$x'_3$};
\node[below] at (2.5,-.5) {$y_3$};\node[below right] at (2.95,.05) {$y_2$};\node[right] at (3.5,.5) {$y_1$};
\node[above] at (2.5,3.5) {$y'_1$};\node[above right] at (2.96,2.94) {$y'_2$};\node[right] at (3.5,2.5) {$y'_3$};
\end{tikzpicture}}
\right)\star\varphi_6$
\caption{$B_{16}(6\times3)$ is the HNN extension of the indicated ``square-type'' RAAG with respect to the isomorphism $\varphi_6\colon F(x_1,x'_1)\to F(y_1,y'_1)$ given by $\varphi_6(x_1)=y_1$ and $\varphi_6(x'_1)=y'_1$. So, the relations $v_6x_1v_6^{-1}=y_1$ and $v_6x'_1v_6^{-1}=y'_1$ hold in $B_{16}(6\times3)$. Note that the subgroup generated by $x_1$ and $x'_1$ (by $y_1$ and $y'_1$, respectively) is free, as the subgraph generated by these two vertices has no edges.}
        \label{p6}
    \end{figure}

\begin{figure}
\centering
$B_{19}(7\times3)\cong\raag\left(\raisebox{-2.75cm}{
\begin{tikzpicture}[x=.79cm,y=.79cm]
\node at (0+.25,0-.25) {\tiny$\bullet$};
\node at (0+.75,0-.75) {\tiny$\bullet$};
\node at (0-.25,0+.25) {\tiny$\bullet$};
\node at (0-.75,0+.75) {\tiny$\bullet$};
\node at (0-.25,4-.25) {\tiny$\bullet$};
\node at (0-.75,4-.75) {\tiny$\bullet$};
\node at (0+.25,4+.25) {\tiny$\bullet$};
\node at (0+.75,4+.75) {\tiny$\bullet$};
\node at (4+.25,4-.25) {\tiny$\bullet$};
\node at (4+.75,4-.75) {\tiny$\bullet$};
\node at (4-.25,4+.25) {\tiny$\bullet$};
\node at (4-.75,4+.75) {\tiny$\bullet$};
\node at (4-.25,0-.25) {\tiny$\bullet$};
\node at (4-.75,0-.75) {\tiny$\bullet$};
\node at (4+.25,0+.25) {\tiny$\bullet$};
\node at (4+.75,0+.75) {\tiny$\bullet$};
\draw(0-.75,.75)--(0-.75,4-.75);
\draw(0-.75,.75)--(0-.25,4-.25);
\draw(0-.75,.75)--(0+.25,4+.25);
\draw(0-.75,.75)--(0+.75,4+.75);
\draw(0-.25,.25)--(0-.25,4-.25);
\draw(0-.25,.25)--(0+.25,4+.25);
\draw(0-.25,.25)--(0+.75,4+.75);
\draw(0+.25,-.25)--(0+.25,4+.25);
\draw(0+.25,-.25)--(0+.75,4+.75);
\draw(0+.75,-.75)--(0+.75,4+.75);
\draw(.75,4.75)--(4-.75,4+.75);
\draw(.75,4.75)--(4-.25,4+.25);
\draw(.75,4.75)--(4+.25,4-.25);
\draw(.75,4.75)--(4+.75,4-.75);
\draw(.25,4.25)--(4-.25,4+.25);
\draw(.25,4.25)--(4+.25,4-.25);
\draw(.25,4.25)--(4+.75,4-.75);
\draw(-.25,4-.25)--(4+.25,4-.25);
\draw(-.25,4-.25)--(4+.75,4-.75);
\draw(-.75,4-.75)--(4+.75,4-.75);
\draw(4.75,4-.75)--(4-.75,-.75);
\draw(4.75,4-.75)--(4-.25,-.25);
\draw(4.75,4-.75)--(4+.25,.25);
\draw(4.75,4-.75)--(4+.75,.75);
\draw(4.25,4-.25)--(4+.25,.25);
\draw(4.25,4-.25)--(4-.25,-.25);
\draw(4.25,4-.25)--(4-.75,-.75);
\draw(4-.25,4-.25)--(4-.25,-.25);
\draw(4-.25,4+.25)--(4-.75,-.75);
\draw(4-.75,4+.75)--(4-.75,-.75);
\draw(4-.75,-.75)--(+.75,-.75);
\draw(4-.75,-.75)--(+.25,-.25);
\draw(4-.75,-.75)--(-.25,.25);
\draw(4-.75,-.75)--(-.75,.75);
\draw(4-.25,-.25)--(+.25,-.25);
\draw(4-.25,-.25)--(-.25,.25);
\draw(4-.25,-.25)--(-.75,.75);
\draw(4.25,.25)--(-.25,.25);
\draw(4.25,.25)--(-.75,.75);
\draw(4.75,.75)--(-.75,.75);
\node[below left] at (.75,-.75) {$x_1$};
\node[below left] at (.25,-.25) {$x_2$};
\node[below left] at (-.25,.25) {$x_3$};
\node[below left] at (-.75,.75) {$x_4$};
\node[above left] at (-.75,4-.75) {$x'_1$};
\node[above left] at (-.25,4-.25) {$x'_2$};
\node[above left] at (.25,4.25) {$x'_3$};
\node[above left] at (.75,4.75) {$x'_4$};
\node[above right] at (4-.75,4+.75) {$y'_1$};
\node[above right] at (4-.25,4+.25) {$y'_2$};
\node[above right] at (4.25,4-.25) {$y'_3$};
\node[above right] at (4.75,4-.75) {$y'_4$};
\node[below right] at (4+.75,.75) {$y_1$};
\node[below right] at (4+.25,.25) {$y_2$};
\node[below right] at (4-.25,-.25) {$y_3$};
\node[below right] at (4-.75,-.75) {$y_4$};
\end{tikzpicture}}
\right)\star\varphi_7$
\caption{The first subgroup under conjugation is generated by $x_1$, $x_2$, $x'_1$ and $x'_2$, while the second one is generated by $y'_1$, $y'_2$, $y_1$ and $y_2$. They are again free due to the lack of edges between these vertices. The isomorphism $\varphi_7\colon F(x_1,x_2,x'_1,x'_2)\to F(y_1,y_2,y'_1,y'_2)$ is now given by $\varphi_7(x_i)=y_{3-i}$ and $\varphi_7(x'_i)=y'_{3-i}$, for $i=1,2$, so the stable generator $v_7$ is related to the rest of the generators via the relations $v_7x_iv_7^{-1}=y_{3-i}$ and $v_7x'_iv_7^{-1}=y'_{3-i}$ $(1\leq i\leq2$).}
\label{p7}
\end{figure}
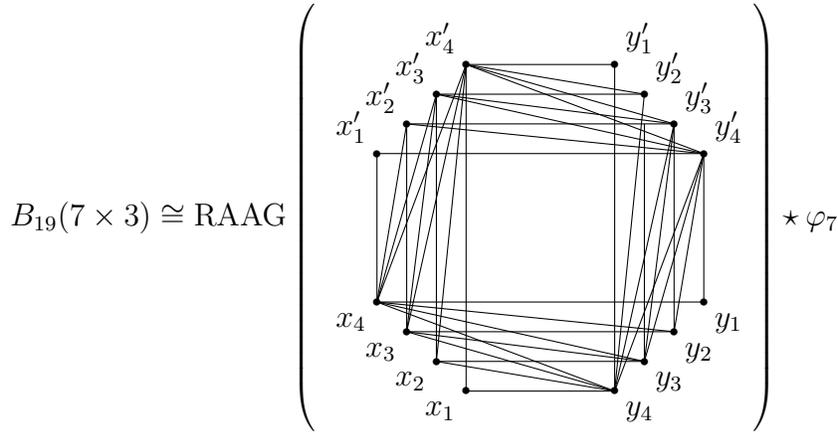

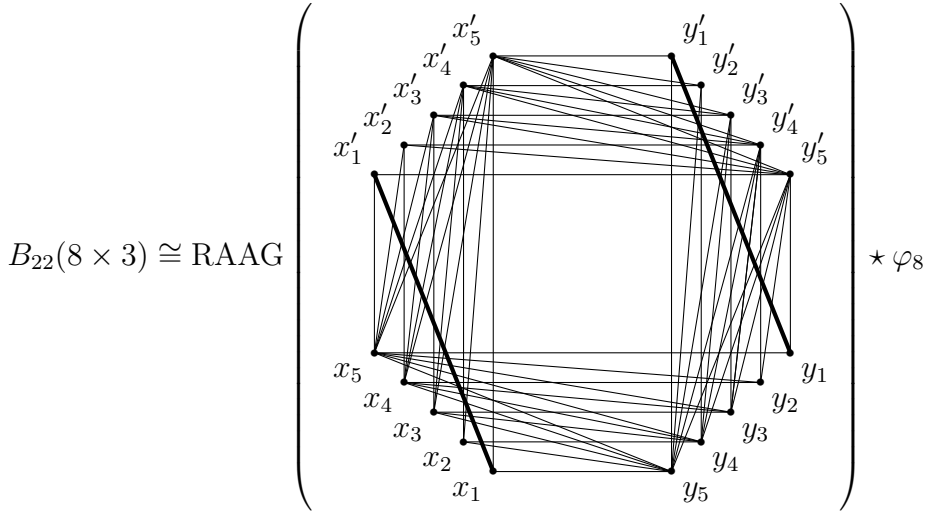
\begin{figure}
\centering
$B_{22}(8\times3)\cong\raag\left(
\raisebox{-3.3cm}{
\begin{tikzpicture}[x=.79cm,y=.79cm]
\node at (0,0) {\tiny$\bullet$};
\node at (0+.5,0-.5) {\tiny$\bullet$};
\node at (0+1,0-1) {\tiny$\bullet$};
\node at (0-.5,0+.5) {\tiny$\bullet$};
\node at (0-1,0+1) {\tiny$\bullet$};
\node at (0,5) {\tiny$\bullet$};
\node at (0+.5,5+.5) {\tiny$\bullet$};
\node at (0+1,5+1) {\tiny$\bullet$};
\node at (0-.5,5-.5) {\tiny$\bullet$};
\node at (0-1,5-1) {\tiny$\bullet$};
\node at (5,0) {\tiny$\bullet$};
\node at (5-.5,0-.5) {\tiny$\bullet$};
\node at (5-1,0-1) {\tiny$\bullet$};
\node at (5+.5,0+.5) {\tiny$\bullet$};
\node at (5+1,0+1) {\tiny$\bullet$};
\node at (5,5) {\tiny$\bullet$};
\node at (5+.5,5-.5) {\tiny$\bullet$};
\node at (5+1,5-1) {\tiny$\bullet$};
\node at (5-.5,5.5) {\tiny$\bullet$};
\node at (5-1,5+1) {\tiny$\bullet$};
\draw(0-1,0+1)--(0-1,5-1);
\draw(0-1,0+1)--(0-.5,5-.5);
\draw(0-1,0+1)--(0,5);
\draw(0-1,0+1)--(0.5,5.5);
\draw(0-1,0+1)--(1,6);
\draw(0-.5,0+.5)--(0-.5,5-.5);
\draw(0-.5,0+.5)--(0,5);
\draw(0-.5,0+.5)--(0.5,5.5);
\draw(0-.5,0+.5)--(1,6);
\draw(0,0)--(0,5);
\draw(0,0)--(0.5,5.5);
\draw(0,0)--(1,6);
\draw(0.5,-0.5)--(0.5,5.5);
\draw(0.5,-0.5)--(1,6);
\draw(1,-1)--(1,6);
\draw(1,6)--(4,6);
\draw(1,6)--(4.5,5.5);
\draw(1,6)--(5,5);
\draw(1,6)--(5.5,4.5);
\draw(1,6)--(6,4);
\draw(.5,5.5)--(4.5,5.5);
\draw(.5,5.5)--(5,5);
\draw(.5,5.5)--(5.5,4.5);
\draw(.5,5.5)--(6,4);
\draw(0,5)--(5,5);
\draw(0,5)--(5.5,4.5);
\draw(0,5)--(6,4);
\draw(0-.5,4.5)--(5.5,4.5);
\draw(0-.5,4.5)--(6,4);
\draw(-1,4)--(6,4);
\draw(6,4)--(6,1);
\draw(6,4)--(5.5,.5);
\draw(6,4)--(5,0);
\draw(6,4)--(4.5,-.5);
\draw(6,4)--(4,-1);
\draw(5.5,4.5)--(5.5,.5);
\draw(5.5,4.5)--(5,0);
\draw(5.5,4.5)--(4.5,-.5);
\draw(5.5,4.5)--(4,-1);
\draw(5.5,4.5)--(5,0);
\draw(5,5)--(5,0);
\draw(5,5)--(4.5,-.5);
\draw(5,5)--(4,-1);
\draw(4.5,5.5)--(4.5,-.5);
\draw(4.5,5.5)--(4,-1);
\draw(4,6)--(4,-1);
\draw(4,-1)--(1,-1);
\draw(4,-1)--(.5,-.5);
\draw(4,-1)--(0,0);
\draw(4,-1)--(-.5,.5);
\draw(4,-1)--(-1,1);
\draw(4.5,-.5)--(.5,-.5);
\draw(4.5,-.5)--(0,0);
\draw(4.5,-.5)--(-.5,.5);
\draw(4.5,-.5)--(-1,1);
\draw(5,0)--(0,0);
\draw(5,0)--(-.5,.5);
\draw(5,0)--(-1,1);
\draw(5.5,0.5)--(-.5,.5);
\draw(5.5,0.5)--(-1,1);
\draw(6,1)--(-1,1);
\draw[ultra thick](1,-1)--(-1,4);\draw[ultra thick](6,1)--(4,6);
\node[below left] at (1,-1) {$x_1$};
\node[below left] at (.5,-.5) {$x_2$};
\node[below left] at (0,0) {$x_3$};
\node[below left] at (-.5,.5) {$x_4$};
\node[below left] at (-1,1) {$x_5$};
\node[above left] at (-1,4) {$x'_1$};
\node[above left] at (-.5,4.5) {$x'_2$};
\node[above left] at (0,5) {$x'_3$};
\node[above left] at (0.5,5.5) {$x'_4$};
\node[above left] at (1,6) {$x'_5$};
\node[above right] at (4,6) {$y'_1$};
\node[above right] at (4.5,5.5) {$y'_2$};
\node[above right] at (5,5) {$y'_3$};
\node[above right] at (5.5,4.5) {$y'_4$};
\node[above right] at (6,4) {$y'_5$};
\node[below right] at (6,1) {$y_1$};
\node[below right] at (5.5,.5) {$y_2$};
\node[below right] at (5,0) {$y_3$};
\node[below right] at (4.5,-.5) {$y_4$};
\node[below right] at (4,-1) {$y_5$};
\end{tikzpicture}}
\right)\star\varphi_8$
\caption{The domain of the isomorphism $\varphi_8$ defining the HNN extension is generated by $x_1,x_2,x_3,x'_1,x'_2,x'_3$, while the codomain is generated by the corresponding ``$y$" elements. The isomorphism is determined by $\varphi_8(x_i)=y_{4-i}$ and $\varphi_8(x'_i)=y'_{4-i}$, for $1\leq i\leq3$, which correspond to relations $v_8x_iv_8^{-1}=y_{4-i}$ and $v_8x'_iv_8^{-1}=y'_{4-i}$ ($1\leq i\leq3$) in $B_{22}(8\times3)$. These relations together with the stable letter $v_8$ must then be added to the RAAG presentation of the indicated graph to yield a presentation for $B_{22}(8\times3)$. The two ``unexpected'' RAAG-type commutation relations $x_1x'_1=x'_1x_1$ and $y_1y'_1=y'_1y_1$ correspond to the two thick edges. The former one, for instance, is forced from the fact that $v_8x_1v_8^{-1}=y_3$ commutes with $v_8x'_1v_8^{-1}=y'_3$.  
In particular, neither the domain nor the domain of $\varphi_8$ are free; they are instead RAAGs associated to the induced subgraphs.}
\label{p8}
\end{figure}

\begin{theorem}\label{resultadoprincipal}
For $p\geq5$ and $q=3$, there is an isomorphism $$B_{3p-2}(p\times 3)\cong\raag(S_{p-3})\star\varphi_p$$ expressing $B_{3p-2}(p\times 3)$ as the {\em HNN} extension of the {\em RAAG} associated to the graph $S_{p-3}$ with respect to the isomorphism $\varphi_p\colon \text{\em RAAG}(X_{p-3})\to\text{\em RAAG}(Y_{p-3})$, where:
\begin{enumerate}
\item $S_{p-3}$ is the ``square-type" bipartite graph with vertices $x_i$, $x'_i$, $y'_i$ and $y_i$ for $1\leq i\leq p-3$, and the following two types of edges:
\begin{itemize}
\item Whenever $i+j>p-3$, there are four edges: one between $x_i$ and $x'_j$, one between $x'_i$ and $y'_j$, one between $y'_i$ and $y_j$, and one between $y_i$ and $x_j$. 
\item Whenever $i+j<p-5$, there are two edges: one between $x_i$ and $x'_j$, and one between $y'_i$ and $y_j$. 
\end{itemize}
\item $\text{\em RAAG}(X_{p-3})$ and $\text{\em RAAG}(Y_{p-3})$ are the full sub\hspace{.18mm}{\em RAAG}\hspace{-.1mm}s of \hspace{.2mm}$\raag(S_{p-3})$ associated to induced subgraphs $X_{p-3}$ and $Y_{p-3}$ of $S_{p-3}$. Here $X_{p-3}$ is induced by the vertices $x_i$ and $x'_i$ with $1\leq i\leq p-5$, while $Y_{p-3}$ is induced by the vertices $y_i$ and $y'_i$ with $1\leq i\leq p-5$.

\item The isomorphism $\varphi_p\colon\raag(X_{p-3})\to\raag(Y_{p-3})$ is determined by $\varphi_p(x_i)=y_{p-i-4}$ and $\varphi_p(x'_i)=y'_{p-i-4}$, for $1\leq i\leq p-5$. 
\end{enumerate}
\end{theorem}

A classifying space for an HNN extension can be constructed as a mapping torus (\cite[Example 1B.13]{MR1867354}), so:
\begin{corollary}\label{mappingtorus}
Consider the inclusions $\raag(X_{p-3})\stackrel{\hspace{1mm}\iota_X}\longhookrightarrow\raag(S_{p-3}) \stackrel{\iota_Y}\longhookleftarrow \raag(Y_{p-3})$. For $p\geq6$, $\UC(3p-2, p\times3)$ is homotopy equivalent to the mapping-cylinder realization (say, using Salvetti complexes) of the graph of groups 
$$
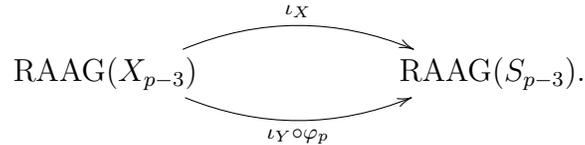
\xymatrix{
\raag(X_{p-3}) \hspace{8mm} \ar@/_1.5pc/[r]_{\iota_Y\circ \varphi_p} \ar@/^1.5pc/[r]^{\iota_X}  & \hspace{8.3mm}\raag(S_{p-3}).
}$$
\end{corollary}

Although geometrically appealing, the 3-dimensional mapping torus in Corollary \ref{mappingtorus} is not efficient homotopicaly speaking, as $\hdim(\UC_{3p-2}(p\times3))\leq2$ (Theorem \ref{resultadocasiprincipal}).

\section{Discrete configuration space}
For a CW complex $X$ with cells $e$, the discrete configuration space $\DF(X,n)$ of $n$ non-colliding labelled cells of $X$ is the subcomplex of $X^n$ consisting of the cells $$(e_1,\ldots,e_n) :=e_1 \times\cdots\times e_n$$ satisfying $\overline{e_i} \cap \overline{e_j}=\varnothing$ for $i\neq j$. The corresponding unordered discrete configuration space $\UDF(X,n)$ is the quotient complex of $\DF(X,n)$ by the cellular free action of the symmetric group $\Sigma_n$ on $n$ letters that permutes cell coordinates. The cell of $\UDF(X,n)$ given by the $\Sigma_n$-orbit of $(e_1,\ldots, e_n)$ is denoted by $\{e_1,\ldots, e_n\}$. For $X=\Gamma$ a graph, this construction was first studied in Abrams' Ph.D. thesis (\cite{MR2701024}).

\begin{theorem}[{\cite{MR2701024}}]\label{Eilenbergmaclane}
For a graph $\Gamma$, $\DF(\Gamma,n)$ and $\UDF(\Gamma,n)$ are aspherical (only the former space might fail to be path-connected, but even so all of its components are aspherical).
\end{theorem}

Let $\Gamma_{p,q}$ be the 1-skeleton of the cube complex $\puz_{p,q}$ obtained by restricting to $[1,p]\times[1,q]$ the standard 2-dimensional cube-complex structure on the plane. See Figure \ref{maxtree}.

\begin{figure}
\centering
\begin{tikzpicture}[x=.8cm,y=.8cm]
\draw (0,0)--(5,0);
\node at (0,0) {$\bullet$};\node at (1,0) {$\bullet$};
\node at (2,0) {$\bullet$};\node at (3,0) {$\bullet$};
\node at (4,0) {$\bullet$};\node at (5,0) {$\bullet$};
\draw (0,1)--(5,1);
\node at (0,1) {$\bullet$};\node at (1,1) {$\bullet$};
\node at (2,1) {$\bullet$};\node at (3,1) {$\bullet$};
\node at (4,1) {$\bullet$};\node at (5,1) {$\bullet$};
\draw (0,2)--(5,2);
\node at (0,2) {$\bullet$};\node at (1,2) {$\bullet$};
\node at (2,2) {$\bullet$};\node at (3,2) {$\bullet$};
\node at (4,2) {$\bullet$};\node at (5,2) {$\bullet$};
\draw (0,3)--(5,3);
\node at (0,3) {$\bullet$};\node at (1,3) {$\bullet$};
\node at (2,3) {$\bullet$};\node at (3,3) {$\bullet$};
\node at (4,3) {$\bullet$};\node at (5,3) {$\bullet$};
\draw (0,0)--(0,3);\draw (1,0)--(1,3);\draw (2,0)--(2,3);
\draw (3,0)--(3,3);\draw (4,0)--(4,3);\draw (5,0)--(5,3);
\node at (0,-.5) {\footnotesize$1$};
\node at (1,-.5) {\footnotesize$2$};
\node at (2,-.5) {\footnotesize$3$};
\node at (3,-.5) {\footnotesize$4$};
\node at (4,-.5) {\footnotesize$5$};
\node at (5,-.5) {\footnotesize$6$};
\node at (-.5,0) {\footnotesize$1$};
\node at (-.5,1) {\footnotesize$2$};
\node at (-.5,2) {\footnotesize$3$};
\node at (-.5,3) {\footnotesize$4$};

\end{tikzpicture}
\caption{The graph $\Gamma_{6,4}$. Fill in the 15 loops with solid squares to get the complex $\puz_{6,4}$.% (left) and its maximal tree (right). In this setting $\varepsilon_i^j$ stands for the (dotted) vertical edge outside the maximal tree whose lower vertex has label $i$. Here $1\leq i\leq p(q-1)$ with $i$ not divisible by $p$. We refer to such an edge as a \emph{deleted} one.
}
\label{maxtree}
\end{figure}
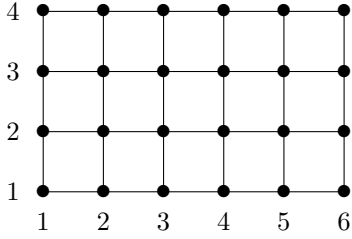

\begin{proposition}[{\cite{MR4640135}}]\label{UXmodel}
$\DF(\puz_{p,q},n)$ sits inside $\C(n,p\times q)$ as a strong $\Sigma_n$-equivariant deformation retract. In particular $\UDF(\puz_{p,q},n)\cong\UC(n,p\times q)$.
\end{proposition}

Observe that $\UDF(\Gamma_{p,q},n)$ is a subcomplex of $\UDF(\puz_{p,q},n)$, and they share a common 1-skeleton. Indeed, any cell $\{e_1,\ldots,e_n\}$ of $\UXnpq$ outside $\UDF(\Gamma_{p,q},n)$ has dimension at least 2, as one of the cell ingredients $e_i$ would have to be an actual square of $\puz_{p,q}$. So, just as $\UDF(\Gamma_{p,q},n)$, $\UXnpq$ is path-connected. In fact 
\begin{equation}\label{soniguales}
\UDF(\Gamma_{p,q},n)=\UXnpq, \text{ which is at most 2-dimensional if $n\geq pq-2$,}
\end{equation}
for then there is no room for a square-type cell ingredient $e_i$, nor for three or more edge-type cell ingredients $e_i$. In particular, Theorem~\ref{Eilenbergmaclane}, Proposition~\ref{UXmodel} and (\ref{soniguales}) give the asphericity of $\UC(n,p\times q)$ noted in Proposition~\ref{asphericalspaces} when $n\geq pq-2$. Likewise, (\ref{soniguales}) and Proposition~\ref{UXmodel} yields the estimate
\begin{equation}\label{alomasdos}
\hdim(\UC(pq-2,p\times q))\leq2
\end{equation}
relevant for Theorem \ref{resultadocasiprincipal} (see also Lemma~\ref{conteocritico} below). The sharpness of this estimate for $\max(p,q)>\min(p,q)\geq3$ will follow once we argue the homological assertions in Theorem~\ref{resultadocasiprincipal}.

\smallskip
The discrete Morse theoretic method in \cite[Theorem 2.5]{MR2171804} was applied in \cite{AGK} to the cubical homotopy model $\UDF(\puz_{p,q},n)$ of $\UC(n,p\times q)$ in order to produce a raw presentation of the fundamental group $B_n(p\times q)$. In the case of interest for us, namely $n=pq-2$ so (\ref{soniguales}) holds, this is done in terms of the gradient field on $\UDF(\Gamma_{p,q},pq-2)$ constructed by Farley-Sabalka with respect to a certain maximal tree $T_{p,q}$ for $\Gamma_{p,q}$. Details are spelled out in \cite[Subsection 2.2]{AGK}, from which Lemma \ref{conteocritico} below follows by direct counting. In addition, the presentation for $B_{pq-2}(p\times q)$ summarized in Proposition \ref{generators} below is easily extracted from the discussion in \cite[Section 4]{AGK} using the notation in \cite[Example 4.5]{AGK}.

\begin{lemma}\label{conteocritico}
For $p\geq q\geq3$, Farley-Sabalka's discrete gradient field on the homotopy model $\UDF(\Gamma_{p,q},pq-2)$ of $\UC(pq-2,p\times q)$ has $c_i$ critical cells of dimension $i$, where $$c_0=1, \quad c_1=3(p-1)(q-1)-2, \quad c_2=\binom{(p-1)(q-1)}2-(p-2)(q-2)$$ and $c_i=0$ for $i\neq0,1,2$.
\end{lemma}

The harmless assumption $p\geq q$ in Lemma \ref{conteocritico}, immaterial for the given conclusion, is relevant for the choosing of the maximal tree $T_{p,q}$ and, thus, for the actual gradient field on $\UDF(\Gamma_{p,q},pq-2)$, as well as for the resulting presentation of $B_{pq-2}(p\times q)$ described next.

\begin{proposition}\label{generators}
For $p\geq q\geq3$, the group $B_{pq-2}(p\times q)$ is presented with generators $a_{\ell,i}$, $b_{\ell,i}$ and $c_{\ell,i}$, for $1\leq \ell\leq q-1$ and $1\leq i\leq p-1$, and relators
\begin{align}
a_{1,1} \text{ and } c_{q-1,1}.& \label{rel2} \\
b_{\ell,i} a_{\ell,j} b_{\ell,i}^{-1} c_{\ell,j}^{-1},& 
\text{ \ for\hspace{.3mm} } 1\leq \ell\leq q-1, \ 1\leq i<j\leq p-1. \label{rel3} \\
c_{\ell,i} a_{\ell+1,j} b_{\ell,i}^{-1} b_{\ell+1,j}^{-1},& \text{ \ for\hspace{.3mm} } 1\leq \ell\leq q-2, \ 1\leq i\leq p-1, \ 1\leq j\leq p-i. \label{rel4} \\
[c_{\ell,i}\cc a_{\ell+1,j}],& \text{ \ for\hspace{.3mm} } 1\leq \ell\leq q-2, \  i,j\in\{3,4,\ldots,p-1\}, \ p+2\leq i+j. \nonumber \\
[c_{\ell,i}\cc a_{\lambda,j}],& \text{ \ for\hspace{.3mm} } 2\leq \ell+1<\lambda\leq q-1, \ i,j\in\{1,2,\ldots,p-1\}. \nonumber
\end{align}
\end{proposition}

\section{Simple commutator-related structure}\label{structure}
Assuming $p\geq q\geq3$, in this section we describe a process of Tietze transformations that leads to a reduction/simplification of generators/relations in the presentation of $B_{pq-2}(p\times q)$ in Proposition~\ref{generators}. Key portions of this process have been suggested by computer exploration.

\smallskip
Use the relators in (\ref{rel4}) with $i=j=1$ and the relators in (\ref{rel3}) with $i=1$ to eliminate, respectively, $c_{\ell,1}$ and $c_{\ell,j}$ for $2\leq j$ via the substitutions
\begin{align}
c_{\ell,1}=b_{\ell+1,1}b_{\ell,1}a_{\ell+1,1}^{-1},& \text{ \ for } 1\leq\ell\leq q-2 \text{ \ (because of (\ref{rel2}), $c_{q-1,1}$ is negligible),} \nonumber \\
c_{\ell,j}=b_{\ell,1}a_{\ell,j}b_{\ell,1}^{-1},& \text{ \ for } 1\leq\ell\leq q-1 \text{ \ and \ } 2\leq j\leq p-1. \label{rep5}
\end{align}
For instance, after applying the substitutions in (\ref{rep5}), relators in (\ref{rel3}) take the form
$$\mbox{
$b_{\ell,i}a_{\ell,j}b_{\ell,i}^{-1}b_{\ell,1}a_{\ell,j}^{-1}b_{\ell,1}^{-1}$ or, after a cyclic shift, $a_{\ell,j}b_{\ell,i}^{-1}b_{\ell,1}a_{\ell,j}^{-1}b_{\ell,1}^{-1}b_{\ell,i}$,
}$$
which can be written down as the commutator $[a_{\ell,j}\cc b_{\ell,i}^{-1} b_{\ell,1} ]$. Note that the latter form of the relator can equivalently be written down as $[a_{\ell,j}\cc b_{\ell,1}^{-1} b_{\ell,i}]$. This type of manipulations will be used without further notice in what follows. We get a presentation of $B_{pq-2}(p\times q)$ with generators $a_{\ell,i}$ and $b_{\ell,i}$ with $1\leq\ell\leq q-1$ and $1\leq i\leq p-1$, and relators
\begin{align}
a_{1,1},& \label{a11} \\
[a_{\ell,j}\cc b_{\ell,1}^{-1} b_{\ell,i}],& 
\text{ \ for\hspace{.3mm} } 1\leq \ell\leq q-1, \ 2\leq i<j\leq p-1, \label{nada} \\
b_{\ell,1} a_{\ell+1,1}^{-1} a_{\ell+1,j} b_{\ell,1}^{-1} b_{\ell+1,j}^{-1} b_{\ell+1,1},& \text{ \ for\hspace{.3mm} } 1\leq \ell\leq q-2, \ 2\leq j\leq p-1, \label{relocho} \\
b_{\ell,1} a_{\ell,i} b_{\ell,1}^{-1} a_{\ell+1,j} b_{\ell,i}^{-1} b_{\ell+1,j}^{-1},& \text{ \ for\hspace{.3mm} } 1\leq \ell\leq q-2, \ 2\leq i\leq p-1, \ 1\leq j\leq p-i, \label{relnueve} \\
[b_{\ell,1} a_{\ell,i} b_{\ell,1}^{-1}\cc a_{\ell+1,j}],& \text{ \ for\hspace{.3mm} } 1\leq \ell\leq q-2, \  i,j\in\{3,4,\ldots,p-1\}, \ p+2\leq i+j, \label{reldiez} \\
[b_{\ell+1,1} b_{\ell,1} a_{\ell+1,1}^{-1}\cc a_{\lambda,j}],& \text{ \ for\hspace{.3mm} } 2\leq \ell+1<\lambda\leq q-1, \ 1\leq j\leq p-1, \label{relonce} \\
[b_{\ell,1} a_{\ell,i} b_{\ell,1}^{-1} \cc a_{\lambda,j}],& \text{ \ for\hspace{.3mm} } 2\leq \ell+1<\lambda\leq q-1, \ 2\leq i\leq p-1, \ 1\leq j\leq p-1. \label{reldoce}
\end{align}

Next, replace the generators $b_{\ell,i}$ having $1\leq\ell\leq q-1$ and $2\leq i\leq p-1$ by corresponding generators $B_{\ell,i}:=b_{\ell,1}^{-1}b_{\ell,i}$. (In terms of Tietze transformations, this means we add generators $B_{\ell,i}$ and relators $B_{\ell,i} b_{\ell,i}^{-1} b_{\ell,1}$, and then eliminate the generators $b_{\ell,i}$ via substitutions $b_{\ell,i}=b_{\ell,1}B_{\ell,i}$.) With these operations, relators with $\ell=j=1$ in (\ref{relnueve}) can be written as $B_{1,i}^{-1} b_{1,1}^{-1} b_{2,1}^{-1} b_{1,1} a_{1,i} b_{1,1}^{-1} a_{2,1}$, which allows us to eliminate generators $B_{1,i}$ via the substitutions
$$
B_{1,i}=b_{1,1}^{-1} b_{2,1}^{-1} b_{1,1} a_{1,i} b_{1,1}^{-1} a_{2,1}, \text{ \ for \ } 2\leq i\leq p-1.
$$
Likewise, relators in (\ref{relocho}) can be written as $b_{\ell,1}a_{\ell+1,1}^{-1}a_{\ell+1,j}b_{\ell,1}^{-1}B_{\ell+1,j}^{-1}$, allowing us to eliminate generators $B_{\ell+1,j}$ via the substitutions
$$
B_{\ell+1,j}=b_{\ell,1}a_{\ell+1,1}^{-1}a_{\ell+1,j}b_{\ell,1}^{-1}, \text{ \ for \ } 1\leq\ell\leq q-2, \ 2\leq j\leq p-1.
$$
We are then left with a presentation of $B_{pq-2}(p\times q)$ with generators $a_{\ell,i}$ and $b_{\ell,1}$ for $1\leq\ell\leq q-1$ and $1\leq i\leq p-1$, and relators
{\footnotesize \begin{align}
[a_{1,j}\cc b_{1,1}^{-1} b_{2,1}^{-1} b_{1,1} a_{1,i} b_{1,1}^{-1} a_{2,1}],& \text{ \ for\hspace{.3mm} } 2\leq i<j\leq p-1, \label{reltrece}\\
[a_{\ell,j}\cc b_{\ell-1,1} a_{\ell,1}^{-1} a_{\ell,i} b_{\ell-1,1}^{-1} ],& \text{ \ for\hspace{.3mm} } 2\leq\ell\leq q-1, \ 2\leq i<j\leq p-1, \label{relcatorce}\\ 
b_{\ell,1} a_{\ell,i} b_{\ell,1}^{-1} a_{\ell+1,1} b_{\ell-1,1} a_{\ell,i}^{-1} a_{\ell,1} b_{\ell-1,1}^{-1} b_{\ell,1}^{-1} b_{\ell+1,1}^{-1},& \text{ \ for\hspace{.3mm} } 2\leq \ell\leq q-2, \ 2\leq i\leq p-1, \label{relquince} \\
b_{1,1} a_{1,i} b_{1,1}^{-1} a_{2,j} a_{2,1}^{-1} b_{1,1} a_{1,i}^{-1} b_{1,1}^{-1} b_{2,1} b_{1,1} a_{2,j}^{-1} a_{2,1} b_{1,1}^{-1} b_{2,1}^{-1},& \text{ \ for\hspace{.3mm} } 2\leq i\leq p-2, \ 2\leq j\leq p-i, \label{reldieciseis} \\
\hspace{-2.76mm}b_{\ell,1} a_{\ell,i} b_{\ell,1}^{-1} a_{\ell+1,j} b_{\ell-1,1} a_{\ell,i}^{-1} a_{\ell,1} b_{\ell-1,1}^{-1} a_{\ell+1,j}^{-1} a_{\ell+1,1} b_{\ell,1}^{-1} b_{\ell+1,1}^{-1}, & \text{ \ for\hspace{.3mm} } 2\leq\ell\leq q-2, \ 2\leq i\leq p-2, \ 2\leq j\leq p-i, \label{reldiecisiete}
\end{align}}\noindent
together with those in (\ref{a11}) and (\ref{reldiez})--(\ref{reldoce}), which have no change under the indicated Tietze transformations. Note that (\ref{reltrece}) and (\ref{relcatorce}) come from (\ref{nada}), whereas (\ref{relquince})--(\ref{reldiecisiete}) come from (\ref{relnueve}). 

\smallskip
For the next round of Tietze transformations, set $u_{\ell}:=b_{\ell,1}$ for $1\leq\ell\leq q-1$ and replace the generators $a_{\ell+1,1}$ ($1\leq\ell\leq q-2$) by corresponding generators $v_{\ell}:=u_{\ell}^{-1} a_{\ell+1,1}$. In addition, set $A_{1,i}:=a_{1,i}$ for $2\leq i\leq p-1$, and replace the generators $a_{\ell,i}$ with $2\leq\ell\leq q-1$ and $2\leq i\leq p-1$ by corresponding generators $A_{\ell,i}:=v_{\ell-1}^{-1} u_{\ell-1}^{-1} a_{\ell,i}$. This yields a presentation for $B_{pq-2}(p\times q)$ with generators
\begin{align}
A_{\ell,i} &\mbox{ for $1\leq\ell\leq q-1$ and $2\leq i\leq p-1$ (because of (\ref{a11}), $a_{1,1}$ is negligible),} \nonumber \\
u_\ell &\mbox{ for $1\leq \ell\leq q-1$, together with $v_\ell$ for $1\leq \ell\leq q-2$, \label{gensuyv}}
\end{align}
and relators
{\footnotesize\begin{align}
[A_{1,j}\cc u_1^{-1} u_2^{-1} u_1 A_{1,i} v_1],& \text{ \ for\hspace{.3mm} } 2\leq i<j\leq p-1, \label{reldieciocho}\\
[u_{\ell-1} v_{\ell-1} A_{\ell,j}\cc u_{\ell-1} A_{\ell,i} u_{\ell-1}^{-1}],& \text{ \ for\hspace{.3mm} } 2\leq\ell\leq q-1, \ 2\leq i<j\leq p-1, \label{reldiecinueve}\\
u_\ell u_{\ell-1} v_{\ell-1} A_{\ell,i} v_\ell u_{\ell-1} A_{\ell,i}^{-1} u_{\ell-1}^{-1} u_\ell^{-1} u_{\ell+1}^{-1},& \text{ \ for\hspace{.3mm} } 2\leq\ell\leq q-2, \ 2\leq i\leq p-1, \label{relveinte}\\
[u_1^{-1} u_2^{-1} u_1 A_{1,i} v_1 \cc A_{2,j}],& \text{ \ for\hspace{.3mm} } 2\leq i\leq p-2, \ 2\leq j\leq p-i, \label{relveintiuno}\\
u_\ell u_{\ell-1} v_{\ell-1} A_{\ell,i} v_\ell A_{\ell+1,j} u_{\ell-1} A_{\ell,i}^{-1} u_{\ell-1}^{-1} A_{\ell+1,j}^{-1} u_\ell^{-1} u_{\ell+1}^{-1},& \text{ \ for\hspace{.3mm} } 2\leq\ell\leq q-2, \ \ 2\leq i\leq p-2, \ 2\leq j\leq p-i, \label{relveintidos}
\end{align}}\noindent
and those coming from (\ref{reldiez})--(\ref{reldoce}) which, after the indicated transformations, become
{\footnotesize\begin{align}
[A_{1,i}\cc v_1 A_{2,j} u_1],& \text{ \ for\hspace{.3mm} } i,j\in\{3,4,\ldots,p-1\}, \ p+2\leq i+j, \label{relveintitres} \\
[u_{\ell-1} v_{\ell-1} A_{\ell,i} \cc v_\ell A_{\ell+1,j} u_\ell],& \text{ \ for\hspace{.3mm} } 2\leq \ell\leq q-2, \  i,j\in\{3,4,\ldots,p-1\}, \ p+2\leq i+j, \label{relveinticuatro} \\
[u_{\ell+1} u_\ell v_\ell^{-1} u_\ell^{-1} \cc u_{\lambda-1} v_{\lambda-1}],& \text{ \ for\hspace{.3mm} } 2\leq \ell+1<\lambda\leq q-1, \label{relveinticinco} \\
[u_{\ell+1} u_\ell v_\ell^{-1} u_\ell^{-1} \cc u_{\lambda-1}v_{\lambda-1}A_{\lambda,j}],& \text{ \ for\hspace{.3mm} } 2\leq \ell+1<\lambda\leq q-1, \ 2\leq j\leq p-1, \label{relveintiseis} \\
[ u_1 A_{1,i} u_1^{-1}\cc u_{\lambda-1} v_{\lambda-1}],& \text{ \ for\hspace{.3mm} } 3\leq\lambda\leq q-1, \ 2\leq i\leq p-1, \label{relveintisiete} \\
[ u_1 A_{1,i} u_1^{-1}\cc u_{\lambda-1}v_{\lambda-1}A_{\lambda,j}],& \text{ \ for\hspace{.3mm} } 3\leq\lambda\leq q-1, \ 2\leq i\leq p-1, \ 2\leq j \leq p-1, \label{relveintiocho} \\
[u_\ell u_{\ell-1} v_{\ell-1} A_{\ell,i} u_\ell^{-1} \cc u_{\lambda-1} v_{\lambda-1}],& \text{ \ for\hspace{.3mm} } 3\leq \ell+1<\lambda\leq q-1, \ 2\leq i\leq p-1, \label{relveintinueve} \\
[u_\ell u_{\ell-1} v_{\ell-1} A_{\ell,i} u_\ell^{-1} \cc u_{\lambda-1}v_{\lambda-1}A_{\lambda,j}],& \text{ \ for\hspace{.3mm} } 3\leq \ell+1<\lambda\leq q-1, \ 2\leq i\leq p-1, \ 2\leq j\leq p-1. \label{reltreinta}
\end{align}}\indent

Note that, in view of (\ref{relveinticinco}), (\ref{relveintisiete}) and (\ref{relveintinueve}), relators (\ref{relveintiseis}), (\ref{relveintiocho}) and (\ref{reltreinta}) simplify, respectively, to $[u_{\ell+1} u_\ell v_\ell^{-1} u_\ell^{-1} \cc A_{\lambda,j}]$, $[ u_1 A_{1,i} u_1^{-1}\cc A_{\lambda,j}]$ and $[u_\ell u_{\ell-1} v_{\ell-1} A_{\ell,i} u_\ell^{-1} \cc A_{\lambda,j}]$. Example~\ref{finq3} below illustrates aspects of the latter presentation for $q=3$, and completes the proof of Theorem~\ref{resultadocasiprincipal} in that case.
 
\begin{example}\label{finq3}{\em
Assume $q=3$, so that there are no relators (\ref{relveinte}) or (\ref{relveintidos}), and the previous description is that of a simple commutator related structure. By direct counting we see that the presentation has $2p-1$ generators and $2(p-3)(p-2)$ relators (note that the sets of relators (\ref{relveinticuatro})--(\ref{reltreinta}) are also empty for $q=3$). The dimension-1 case of the Hurewicz theorem then gives $H_1(\UC_{3p-2}(p\times3))\cong\mathbb{Z}^{2p-1}$, while the discussion around (\ref{soniguales}) shows that $H_2(\UC_{3p-2}(p\times3))$ is torsion-free, say of rank $\beta_2$, and that the Euler characteristic of $\UC_{3p-2}(p\times3)$ is given by $\chi=\beta_2-2(p-1)$. On the other hand, Lemma \ref{conteocritico} and \cite[Corollary 3.5]{MR1358614} show that $\chi=1-(6p-8)+\binom{2(p-1)}{2}-p+2$. The last two equalities yield $\beta_2=2(p-3)(p-2)$ which, as noted above, is the number of relators in the simple commutator-related presentation. Specializing to the case $p=q=3$, we get $B_{7}(3\times3)=F_5$, as asserted in Example \ref{stillraags}.1. In particular $\UC(7,3\times3)$ is homotopy equivalent to a wedge of 5 circles. However, for $p>q=3$, $\hdim(\UC(pq-2,p\times q))=2$, which is forced from (\ref{alomasdos}) and the non-triviality of the corresponding 2-dimensional homology group.
}\end{example}

In the rest of the section we assume $p\geq q\geq4$. By (\ref{relveinte}), relators (\ref{relveintidos}) can be expressed as relations $u_{\ell-1} A_{\ell,i}^{-1} u_{\ell-1}^{-1} = A_{\ell+1,j} u_{\ell-1} A_{\ell,i}^{-1} u_{\ell-1}^{-1} A_{\ell+1,j}^{-1}$ or, equivalently, as commutator-type relators
\begin{equation}\label{reltreintayuno}
A_{\ell+1,j} \cdot (u_{\ell-1} A_{\ell,i}^{-1} u_{\ell-1}^{-1}) \cdot A_{\ell+1,j}^{-1} \cdot(u_{\ell-1} A_{\ell,i}^{-1} u_{\ell-1}^{-1})^{-1},
\end{equation}
holding in the same range of values for $\ell$, $i$ and $j$ as that indicated in (\ref{relveintidos}). We then use a final round of Tietze transformations in order to simplify the only one type of relators that have not yet been written down as commutators, namely relators (\ref{relveinte}). Explicitly, the case $i=2$ of relators (\ref{relveinte}) allows us to recursively eliminate all generators $u_{\ell+1}$ with $\ell=q-2,q-3,\ldots,2$ (in that order) via the substitutions
\begin{equation}\label{lastround}
u_{\ell+1}=u_\ell u_{\ell-1} v_{\ell-1} A_{\ell,2} v_\ell u_{\ell-1} A_{\ell,2}^{-1} u_{\ell-1}^{-1} u_\ell^{-1}.
\end{equation}
Any such substitution applied on a relator that has already been written down as a commutator does not change the commutator characteristic\footnote{The resulting substituted commutator might be trivial, i.e., of the form $[a,a]$. But the minimality of the final presentation (Corollary \ref{eficiencia}) will prevent such a possibility.}. Furthermore, for a fixed value of $\ell\in\{2,3,\ldots,q-2\}$, direct verification gives that a relator $u_\ell u_{\ell-1} v_{\ell-1} A_{\ell,i} v_\ell u_{\ell-1} A_{\ell,i}^{-1} u_{\ell-1}^{-1} u_\ell^{-1} u_{\ell+1}^{-1}$ with $3\leq i\leq p-1$ is transformed under the substitution in (\ref{lastround}) into (the $u_\ell u_{\ell-1} v_{\ell-1}$ conjugate of) the relator $A_{\ell,i} v_\ell u_{\ell-1}A_{\ell,i}^{-1} A_{\ell,2} u_{\ell-1}^{-1} v_\ell^{-1} A_{\ell,2}^{-1}$ or, equivalently, the commutator-type relator
$$
v_\ell u_{\ell-1}\cdot A_{\ell,i}^{-1} A_{\ell,2}\cdot u_{\ell-1}^{-1} v_\ell^{-1}\cdot A_{\ell,2}^{-1} A_{\ell,i}.
$$
Consequently this last round of Tietze transformations renders a simple commutator-related presentation of $B_{pq-2}(p\times q)$ having $\beta_1:=(p-1)(q-1)+1$ generators, namely $u_1$, $u_2$ and $v_\ell$ for $1\leq \ell\leq q-2$, as well as $A_{\ell,i}$ for $1\leq\ell\leq q-1$ and $2\leq i\leq p-1$, and by direct counting
$$\beta_2:=\frac{p^2q^2+p^2+q^2-2pq^2-2p^2q-3pq+7p+7q-6}{2}>0$$
relators (see Table \ref{conteodirecto}, where the reported number of relators of type (\ref{relveinte}) takes into account the relators that are eliminated during the last round of Tietze transformations). The fact that $\beta_1$ and $\beta_2$ agree with the corresponding Betti numbers of $\UC_{pq-2}(p\times q)$ follows from an arithmetic verification that is identical to the one used in Example \ref{finq3}. This completes the proof of Theorem \ref{resultadocasiprincipal}.

\begin{table}[h!]
\caption{Number of relators by type}
\label{conteodirecto}
\hspace{1cm}\begin{tabular}{|c|c|c|c|c|c|}
\hline
type & (\ref{reldieciocho}) & (\ref{reldiecinueve}) & (\ref{relveinte}) & (\ref{relveintiuno}) & (\ref{relveintidos})=(\ref{reltreintayuno})  \\ \hline
\raisebox{1mm}{amount} & 
\rule{0mm}{6mm}\raisebox{1mm}{$\frac{(p-3)(p-2)}2$} & 
\rule{0mm}{6mm}\raisebox{1mm}{$\frac{(p-3)(p-2)(q-2)}2$} & 
\rule{0mm}{6mm}\raisebox{1mm}{$(p-3)(q-3)$} & 
\rule{0mm}{6mm}\raisebox{1mm}{$\frac{(p-3)(p-2)}2$} & 
\rule{0mm}{6mm}\raisebox{1mm}{$\frac{(p-3)(p-2)(q-3)}2$} \\
\hline
\end{tabular}

\medskip
\hspace{1cm}\begin{tabular}{|c|c|c|c|c|c|}
\hline 
type & (\ref{relveintitres}) & (\ref{relveinticuatro}) & (\ref{relveinticinco}) & (\ref{relveintiseis}) & (\ref{relveintisiete}) \\ \hline
\raisebox{1mm}{amount} & 
\rule{0mm}{6mm}\raisebox{1mm}{$\frac{(p-3)(p-2)}2$} & 
\rule{0mm}{6mm}\raisebox{1mm}{$\frac{(p-3)(p-2)(q-3)}2$} &  
\rule{0mm}{6mm}\raisebox{1mm}{$\frac{(q-3)(q-2)}2$} &  
\rule{0mm}{6mm}\raisebox{1mm}{$\frac{(p-2)(q-3)(q-2)}2$} &
\rule{0mm}{6mm}\raisebox{1mm}{$(p-2)(q-3)$} \\
\hline
\end{tabular}

\medskip
\hspace{1cm}\begin{tabular}{|c|c|c|c|}
    \hline 
    type & (\ref{relveintiocho}) & (\ref{relveintinueve}) & (\ref{reltreinta}) \\ \hline
\raisebox{1mm}{amount} & 
\rule{0mm}{6mm}\raisebox{1mm}{$(p-2)^2(q-3)$} &
\rule{0mm}{6mm}\raisebox{1mm}{$\frac{(p-2)(q-4)(q-3)}2$} &  
\rule{0mm}{6mm}\raisebox{1mm}{$\frac{(p-2)^2(q-4)(q-3)}2$}\\
\hline
\end{tabular}
\end{table}

\section{Commutators and conjugations: The HNN extension}
In this section we organize the presentation of $B_{3p-2}(p\times3)$ in Example \ref{finq3}. Start by using the simplified notation $u$ and $v$ for the generators $u_1$ and $v_1$ in (\ref{gensuyv}). Additionally, replace the generator $u_2$ by a new generator $w:=u^{-1}u_2u$. In these terms, relators (\ref{reldieciocho})--(\ref{reltreinta}) simplify to the four types of relators
\begin{align}
[w^{-1} A_{1,i} v \cc A_{1,j}],& \text{ \ for\hspace{.3mm} } 2\leq i<j\leq p-1, \label{treintaycinco} \\
[w^{-1} A_{1,i} v \cc A_{2,j}],& \text{ \ for\hspace{.3mm} } 2\leq i\leq p-2, \ 2\leq j\leq p-i, \label{treintayseis} \\
[vA_{2,j}u\cc A_{2,i}],& \text{ \ for\hspace{.3mm} } \ 2\leq i<j\leq p-1, \label{cjgd}\\
[v A_{2,j} u \cc A_{1,i}],& \text{ \ for\hspace{.3mm} } i,j\in\{3,4,\ldots,p-1\}, \ p+2\leq i+j, \label{treintaysiete}
\end{align}
where relators in (\ref{cjgd}) have been written down as $u$-conjugates of those in (\ref{reldiecinueve}). For example, with $p=4$, this presentation yields the first assertion in Example \ref{stillraags}.2 after the generator $A_{1,2}$ is replaced by $A'_{1,2}:=w^{-1}A_{1,2}v$, and the generator $A_{2,3}$ is replaced by $A'_{2,3}:=vA_{2,3}u$. Actually, since the first assertion in Example \ref{stillraags}.1 is covered by Example \ref{finq3}, we will assume $p\geq5$ throughout this section. Additionally, for elements $x$ and $y$ of a given group, we will use the notation $x\,\&\,y$ as a substitute of the commutation relation $[x,y]=1$. Likewise, when the notation $x\,\&\,y$ is used as a relator, we mean $[x,y]$.

\smallskip
The easy verification of the following result is left as an exercise for the reader.
\begin{lemma}\label{changebasis}
Assume $a_2,a_3,\ldots,a_{k-1}$ and $b_2,b_3\ldots,b_{k-1}$ are elements of a given group and set $a_1=1=b_k$. For $2\leq i\leq k-1$, choose elements 
$$
\mbox{$A_i\in\{a_i^{\pm1},a_{i-1}^{\pm1}a_i^{\pm1}, a_i^{\pm1}a_{i-1}^{\pm1}\}$ \ and \ \hspace{.3mm}$B_i\in\{b_i^{\pm1},b_{i+1}^{\pm1}b_i^{\pm1}, b_i^{\pm1}b_{i+1}^{\pm1}\}$.}
$$
Then $a_i\,\&\,b_j$ for $2\leq i<j\leq k-1$ if and only if $A_i\,\&\,B_j$ for $2\leq i<j\leq k-1$.
\end{lemma}

We apply Lemma \ref{changebasis} with $k=p$ in the following situations.
\begin{itemize}
\item Set $a_i:=w^{-1}A_{1,i}v$ and $b_j:=A_{1,j}$, so that relators (\ref{treintaycinco}) can be written in the form
\begin{equation}\label{39}
\begin{aligned}
w^{-1}A_{1,2}v \;\;\&\;\; A_{1,j}^{-1}A_{1,j+1}, & \text{ \ for } 2<j\leq p-1,\\ v^{-1}A_{1,i-1}^{-1}A_{1,i}v \;\;\&\;\;A_{1,j}^{-1}A_{1,j+1}, & \text{ \ for } 3\leq i<j\leq p-1.
\end{aligned}
\end{equation}
\item Relators (\ref{treintayseis}) can be written in the form $[w^{-1}A_{1,i}v\cc A_{2,p+1-j}]$ for $2\leq i<j\leq p-1$ or, by setting $a_i:=w^{-1}A_{1,i}v$ and $b_j:=A_{2,p+1-j}$, in the form
\begin{equation}\label{40}
\begin{aligned}
w^{-1}A_{1,2}v \;\;\&\;\; A_{2,p+1-j}A_{2,p-j}^{-1}, & \text{ \ for } 2<j\leq p-1,\\ v^{-1}A_{1,i-1}^{-1}A_{1,i}v \;\;\&\;\;A_{2,p+1-j}A_{2,p-j}^{-1}, & \text{ \ for } 3\leq i<j\leq p-1.
\end{aligned}
\end{equation}
\item Set $a_i:=A_{2,i}$ and $b_j:=vA_{2,j}u$, so that relators (\ref{cjgd}) can be written in the form
\begin{equation}\label{41}
\begin{aligned}
A_{2,i}A_{2,i-1}^{-1} \;\;\&\;\; v A_{2,p-1}u, \hspace{1.075cm}& \text{ \ for } 2\leq i<p-1,\\ A_{2,i}A_{2,i-1}^{-1} \;\;\&\;\; v A_{2,j+1}A_{2,j}^{-1}v^{-1}, & \text{ \ for } 2\leq i<j\leq p-2.
\end{aligned}
\end{equation}
\item Relators (\ref{treintaysiete}) can be written in the form $[vA_{2,p+1-i}u\cc A_{1,j}]$ for $2\leq i<j\leq p-1$ or, by setting $a_i:=vA_{2,p+1-i}u$ and $b_j:=A_{1,j}$, in the form
\begin{equation}\label{42}
\begin{aligned}
vA_{2,p-1}u \;\;\&\;\; A_{1,j}^{-1}A_{1,j+1}, & \text{ \ for } 2<j\leq p-1,\\ vA_{2,p+2-i}A_{2,p+1-i}^{-1}v^{-1} \;\;\&\;\;A_{1,j}^{-1}A_{1,j+1}, & \text{ \ for } 3\leq i<j\leq p-1.
\end{aligned}
\end{equation}
\end{itemize}
As indicated in Lemma \ref{changebasis}, here we set $A_{1,p}=1=A_{2,1}$. All together, we get:

\begin{corollary}\label{abcd}
For $1\leq i\leq p-3$, consider the elements $\alpha(i)$, $\beta(i)$, $\gamma(i)$ and $\delta(i)$ given by
\begin{itemize}
\item $\alpha(i):=A_{2,p-i-1}A_{2,p-i-2}^{-1}$ \ \text{\em{(recall $A_{2,1}=1$),}}
\item $\beta(i):=\begin{cases}
w^{-1}A_{1,2}v, & i=p-3; \\
v^{-1}A_{1,p-i-2}^{-1}A_{1,p-i-1}v, & 1\leq i\leq p-4,
\end{cases}$
\item $\gamma(i):=A_{1,i+2}^{-1}A_{1,i+3}$  \ \text{\em{(recall $A_{1,p}=1$),}}
\item $\delta(i):=\begin{cases}
vA_{2,p-1}u, & i=p-3; \\
vA_{2,i+3}A_{2,i+2}^{-1}v^{-1}, & 1\leq i\leq p-4.
\end{cases}$
\end{itemize}
Then $B_{3p-2}(p\times3)$ is presented by generators $u$, $v$, $w$, $A_{1,i}$ and $A_{2,i}$ for $2\leq i\leq p-1$, and commutation relations $\alpha(i) \;\&\; \beta(j)$, $\beta(i) \;\&\; \gamma(j)$, $\gamma(i) \;\&\; \delta(j)$ and $\delta(i) \;\&\; \alpha(j)$ holding for $i,j\in\{1,2,\ldots,p-3\}$ whenever $i+j>p-3$.
\end{corollary}
\begin{proof}
The asserted commutator relations $\alpha(i) \,\&\, \beta(j)$, $\beta(i) \,\&\, \gamma(j)$, $\gamma(i) \,\&\, \delta(j)$ and $\delta(i) \,\&\, \alpha(j)$ are those in (\ref{40}), (\ref{39}), (\ref{42}) and (\ref{41}), respectively.
\end{proof}

\begin{lemma}\label{auxi}
The following relations hold in $B_{3p-2}(p\times3)$\emph{:}
\begin{enumerate}[(i)]
\item\label{eqi}
\ $\alpha(i)\alpha(i+1)\cdots\alpha(p-3)=A_{2,p-i-1}$, for $1\leq i\leq p-3$.
\item\label{eqii}
\ $v^{-1}\delta(p-4)v\alpha(1)\alpha(2)\cdots\alpha(p-3) =A_{2,p-1}$.
\item\label{eqiii}
\ $\gamma(p-3)^{-1} \gamma(p-4)^{-1} \cdots \gamma(i)^{-1}=A_{1,i+2}$, for $1\leq i\leq p-3$.
\item\label{eqiv}
\ $\gamma(p-3)^{-1}\gamma(p-4)^{-1}\cdots \gamma(1)^{-1}v\beta(p-4)^{-1}v^{-1}=A_{1,2}$.
\item\label{eqv}
\ $\alpha(p-3)^{-1}\alpha(p-4)^{-1}\cdots\alpha(1)^{-1}v^{-1}\delta(p-4)^{-1}\delta(p-3)=u$.
\item\label{eqvi}
\ $\gamma(p-3)^{-1}\gamma(p-4)^{-1}\cdots\gamma(1)^{-1}v\beta(p-4)^{-1}\beta(p-3)^{-1}=w$.
\item\label{conjugaciones1}
\ $v \beta(p-i-4) v^{-1}=\gamma(i)$, for $1\leq i\leq p-5$.
\item\label{conjugaciones2}
\ $v \alpha(p-i-4) v^{-1}=\delta(i)$, for $1\leq i\leq p-5$.
\end{enumerate}
\end{lemma}

Extensive calculations based on Tietze transformations led the authors to realize that the commutation relations in Corollary \ref{abcd} and relations {\em(\ref{conjugaciones1})} and {\em(\ref{conjugaciones2})} in Lemma \ref{auxi} capture a description of $B_{3p-2}(p\times3)$ as an HNN RAAG-extension. For the sake of brevity, such a fact will be proved below through a direct argument that avoids the use of a long sequence of Tietze transformations. In what follows $H_p$ stands for the group presented by generators $V$, $A(i)$, $B(i)$, $C(i)$ and $D(i)$, for $1\leq i\leq p-3$, and relators/relations
\begin{align}
&\mbox{\small $A(i) \;\&\; B(j)$, $B(i) \;\&\; C(j)$, $C(i) \;\&\; D(j)$, $D(i) \;\&\; A(j)$, for $i,j\in\{1,\ldots,p-3\}$, $i+j>p-3,$} \label{lasdeconmutacion} \\
&\mbox{\small $V B(p-i-4) V^{-1}=C(i)$ and $V A(p-i-4) V^{-1}=D(i)$, for $1\leq i\leq p-5$.}\label{lasdeconjugacion}
\end{align}

\begin{theorem}\label{theta}
The group morphism $\theta\colon H_p\to B_{3p-2}(p\times3)$ given by $\theta(V)=v$, $\theta(A(i))=\alpha(i)$, $\theta(B(i))=\beta(i)$, $\theta(C(i))=\gamma(i)$ and $\theta(D(i))=\delta(i)$ for $1\leq i\leq p-3$ is a well-defined isomorphism.
\end{theorem}
\begin{proof}
Corollary \ref{abcd} and relations {\em(\ref{conjugaciones1})} and {\em(\ref{conjugaciones2})} in Lemma \ref{auxi} show that $\theta$ is well-defined. We show that $\theta$ is an isomorphism by constructing its inverse. Consider the elements of $H_p$
\begin{align}
\Theta(A_{2,p-i-1})&:=A(i)A(i+1)\cdots A(p-3), \text{ for }1\leq i\leq p-3, \label{45} \\
\Theta(A_{2,p-1})&:=V^{-1}D(p-4)VA(1)A(2)\cdots A(p-3), \\
\Theta(A_{1,i+2})&:=C(p-3)^{-1}C(p-4)^{-1}\cdots C(i)^{-1}, \text{ for }1\leq i\leq p-3, \\
\Theta(A_{1,2})&:=C(p-3)^{-1}C(p-4)^{-1}\cdots C(1)^{-1}VB(p-4)^{-1}V^{-1}, \\
\Theta(u)&:=A(p-3)^{-1}A(p-4)^{-1}\cdots A(1)^{-1}V^{-1}D(p-4)^{-1}D(p-3), \\
\Theta(w)&:=C(p-3)^{-1}C(p-4)^{-1}\cdots C(1)^{-1}VB(p-4)^{-1}B(p-3)^{-1}, \\
\Theta(v)&:=V. \label{51}
\end{align}
Relations (\ref{(uno)})--(\ref{(ocho)}) below follow directly from definitions (\ref{45})--(\ref{51}), with relations (\ref{lasdeconjugacion}) relevant for the verification of (\ref{(cuatro)}) and (\ref{(ocho)}) when $i\leq p-5$. 
\begin{enumerate}[(I)]
\item\label{(uno)} \ $\Theta(A_{2,p-i-1})\Theta(A_{2,p-i-2})^{-1}=A(i)$, for $1\leq i\leq p-4$.
\item\label{(dos)} \ $\Theta(A_{2,2})=A(p-3)$.
\item\label{(tres)} \ $\Theta(w)^{-1} \Theta(A_{1,2})\Theta(v)=B(p-3)$.
\item\label{(cuatro)} \ $\Theta(v)^{-1}\Theta(A_{1,p-i-2})^{-1}\Theta(A_{1,p-i-1})\Theta(v)=B(i)$, for $1\leq i\leq p-4$.
\item\label{(cinco)} \ $\Theta(A_{1,i+2})^{-1}\Theta(A_{1,i+3})=C(i)$, for $1\leq i\leq p-4$.
\item\label{(seis)} \ $\Theta(A_{1,p-1})^{-1}=C(p-3)$.
\item\label{(siete)} \ $\Theta(v)\Theta(A_{2,p-1})\Theta(u)=D(p-3)$.
\item\label{(ocho)} \ $\Theta(v)\Theta(A_{2,i+3})\Theta(A_{2,i+2})^{-1}\Theta(v)^{-1}=D(i)$, for $1\leq i\leq p-4$.
\end{enumerate}
The point then is that Corollary \ref{abcd}, relations (\ref{lasdeconmutacion}) and relations (\ref{(uno)})--(\ref{(ocho)}) imply that elements (\ref{45})--(\ref{51}) determine a well defined morphism $\Theta\colon B_{3p-2}(p\times3)\to H_p$ sending $\alpha(i)$, $\beta(i)$, $\gamma(i)$ and $\delta(i)$, for $1\leq i\leq p-3$, into $A(i)$, $B(i)$, $C(i)$ and $D(i)$, respectively. In particular $\Theta\circ\theta=1$. On the other hand, the equality $\theta\circ\Theta=1$ follows directly from relations \emph{(\ref{eqi})--(\ref{eqvi})} in Lemma~\ref{auxi}. 
\end{proof}

We close the section by noticing that the group $H_p$ is an HNN extension of a RAAG. Let $G$ be a group presented through a set of generators $\mathcal{G}$ and a set of relators $\mathcal{R}$, and assume that $H_1$ and $H_2$ are isomorphic subgroups of $G$. Choose an isomorphism $\varphi\colon H_1\to H_2$ and let $v\not\in G$ be a new generator symbol. The HNN extension of $G$ with respect to $\varphi$, denoted by $G\star\varphi$, is the group presented through generators $\mathcal{G}\cup\{v\}$ and relators $\mathcal{R}\cup\{vhv^{-1}=\varphi(h)\colon h\in H_1\}$. As shown in \cite{MR32641}, the obvious map $G\to G\star\varphi$ is a monomorphism. Note that, by construction, the resulting subgroups $H_1$ and $H_2$ of $G\star\varphi$ are $v$-conjugated (in $G\star\varphi$).

\smallskip
In dealing with HNN extensions of RAAGs, the following fact proves to be handy:
\begin{lemma}\label{EHNN}
Let $G$ be an induced subgraph of the graph $\Gamma$, and let $H$ be the subgroup of $\raag(\Gamma)$ generated by the vertices of $G$. Then $H=\raag(G)$.
\end{lemma}
\begin{proof}
This is a standard property and we include proof details for completeness. The graph inclusion $G\hookrightarrow\Gamma$ determines a group morphism $\iota\colon\raag(G)\to\raag(\Gamma)$. Since $G$ is an induced subgraph of $\Gamma$, the rule
$$
v\mapsto\begin{cases}
v, & \text{ $v$ is vertex of $G$;} \\
1, & \text{ otherwise,}\end{cases}
$$
determines a group morphism $\pi\colon\raag(\Gamma)\to\raag(G)$ satisfying $\pi\circ\iota=1$. In particular $\iota$ is monic and sets the asserted isomorphism $\raag(G)\cong\text{Im}(\iota)=H$.
\end{proof}

Let $\mathcal{S}_{p-3}$ be the graph with vertices $A(i)$, $B(i)$, $C(i)$ and $D(i)$ for $1\leq i\leq p-3$ and, whenever $i+j>p-3$, four edges, one between $A(i)$ and $B(j)$, one between $B(i)$ and $C(j)$, one between $C(i)$ and $D(j)$, and one between $D(i)$ and $A(j)$. Of course, in RAAG notation, these edges account for relators (\ref{lasdeconmutacion}). For instance, since the set of relators (\ref{lasdeconjugacion}) is empty when $p=5$, Theorem \ref{theta} gives $B_{13}(5\times3)\cong H_5=\raag(\mathcal{S}_{2}+1)$, which recovers the first assertion in Example \ref{stillraags}.3 as well as Theorem \ref{resultadoprincipal} for $p=5$. The situation for $p\geq6$ will be fully parallel once conjugation relations (\ref{lasdeconjugacion}) are taken into account through the use of a suitable HNN extension.

\begin{proof}[Proof of Theorem \ref{resultadoprincipal}]
Let $\mathcal{X}_{p-3}$ (respectively, $\mathcal{Y}_{p-3}$) be the induced subgraph of $\mathcal{S}_{p-3}$ generated by vertices $A(i)$ and $B(i)$ for $1\leq i\leq p-5$ (respectively, $C(i)$ and $D(i)$ for $1\leq i\leq p-5$). Consider the bijection of vertices $\varphi\colon V(\mathcal{X}_{p-3})\to V(\mathcal{Y}_{p-3})$ given by
$$%\begin{equation}\label{correct}
\mbox{$\varphi (B(p-i-4))= C(i)$ \,and \,$\varphi (A(p-i-4))= D(i)$, \,for $1\leq i\leq p-5$.}
$$%\end{equation}
Turn $\varphi$ into an isomorphism of graphs by adding to $\mathcal{X}_{p-3}$ and to $\mathcal{S}_{p-3}$ (respectively, to $\mathcal{Y}_{p-3}$ and to $\mathcal{S}_{p-3}$) an edge between $A(i)$ and $B(j)$ (respectively, between $C(i)$ and $D(j)$) whenever $i+j<p-5$. Let $S_{p-3}$, $X_{p-3}$ and $Y_{p-3}$ be the resulting graph and induced subgraphs. The proof conclusion then follows from Theorem \ref{theta} by noticing that the new edges amount to commutativity relations $[A(i),B(j)]=1$ and $[C(i),D(j)]=1$, for $i+j<p-5$ which, at any rate, are forced by (\ref{lasdeconjugacion}).
\end{proof}

\section{Epilogue: (Not) Highly subdivided graphs}
Undoubtedly, discrete configuration spaces are foundational in our results. Yet, the graph we are interested in, namely $\Gamma_{p,q}$, is far from being sufficiently subdivided, which means that the homotopy type of the discrete configuration space $\UDF(\Gamma_{p,q},n)$ is not guaranteed to agree with that of the regular (point-wise) configuration space $\UF(\Gamma_{p,q},n)$ (see \cite[Theorem 2.1]{MR2701024}). Actually both topologies differ in general. For instance, for $p\geq2$, the first Betti number of $\UDF(\Gamma_{p\times2},2p-2)$ is $2p-3$ (Remark \ref{q2case}), whereas the results in \cite[Subsection 3.2]{MR3000570} yield $H_1(\UF(\Gamma_{p\times2},2p-2))=\mathbb{Z}$.

\smallskip
Despite intrinsic differences between discrete and point-wise configuration space models, the form of our results fits within general properties of (classical) graph braid groups. For instance, on a computational perspective, it is striking to remark that the first Betti number reported in Theorem~\ref{resultadocasiprincipal} agrees with that found by Ko-Park in \cite[Subsection 3.2]{MR3000570} for the point-wise configuration space $\UF(\Gamma_{p,q},n)$ for \emph{any} $n\geq2$. Theoretical coincidences are also notable. For one, $H_*(\UDF(\Gamma_{p,q},pq-2))$ is torsion-free, which matches the point-wise fact proved in \cite[Corollary~3.6]{MR3000570} for planar graphs (such as $\Gamma_{p,q}$). Likewise, we have proved that $\pi_1(\UDF(\Gamma_{p,q},pq-2))$ is a simple commutator-related group, which aligns with \cite[Conjecture 4.9]{MR3000570}. Furthermore, closely related to the simple commutator-related property is the possibility of having a graphical description, as disjoint circuits, of the factors in commutator relators. Such a situation has been shown to hold for $\UF(\Gamma,2)$ when $\Gamma$ is planar (see \cite[Theorem 4.8]{MR3000570} and \cite[Theorem 5.6 and Conjecture 5.7]{MR2949126}). This property also arises for the simple commutator-related group $B_{pq-2}(p\times q)$. Details will be spelled out elsewhere.

\smallskip
The similarities/differences between discrete and point-wise configuration spaces make it clear that the study of the former ones, whether or not their homotopy type agree with that of the latter ones, is an interesting and rich subject of research. See for instance \cite{MR3003699}.

%\bibliographystyle{plain}
%\bibliography{bib}

{\sc \ 

Mathematics Department

Center for Research and Advanced Studies

Av.~I.P.N n\'umero~2508, San Pedro Zacatenco, Mexico City 07000, Mexico.}

\tt oalvarado@math.cinvestav.mx

jesus.glz-espino@cinvestav.mx and jesus@math.cinvestav.mx

\end{document}